\documentclass[10pt]{amsart}
\usepackage{amsthm}
\usepackage{amsmath}
\usepackage{amssymb}
\usepackage{latexsym}
\usepackage{bbm}
\usepackage[dvips]{graphicx}
\usepackage{enumerate}
\usepackage{verbatim}
\usepackage{amsfonts}
\usepackage{epsfig}
\usepackage{tikz}
\usetikzlibrary{fit}
\usetikzlibrary{shapes}
\usepackage{mathdots}
\usepackage{url}
\usepackage{subfigure}

\newcommand{\R}{\mathbb{R}}

\newtheorem{thm}{Theorem}[section]

\newtheorem{lem}[thm]{Lemma}
\newtheorem{cor}[thm]{Corollary}

\theoremstyle{definition}
\newtheorem{defn}[thm]{Definition}

\newtheorem*{que}{Question}

\begin{document}

\title{A Decomposition Theorem for Higher Rank Coxeter Groups}
\author{Ryan Blair}
\author{Ryan Ottman}
\thanks{Research partially supported by an NSF grant.}

\begin{abstract}
In this paper, we show that any Coxeter graph which defines a higher rank Coxeter group must have disjoint induced subgraphs each of which defines a hyperbolic or higher rank Coxeter group. We then use this result to demonstrate several classes of Coxeter graphs which define hyperbolic Coxeter groups.
\end{abstract} \maketitle

\section{Introduction}
Every Coxeter group can be faithfully represented as a group generated by reflections on a suitably chosen metric vector space, where a \emph{metric vector space} is a finite-dimensional real vector space $V$ with an associated symmetric bilinear form and a \emph{reflection} is a non-trivial form-preserving linear transformation that fixes a codimension one subspace. Given a Coxeter group $G$, one can ask on which metric vector spaces $G$ acts faithfully as a group generated by reflections.  Symmetrically, given a metric vector space $V$, one can ask which Coxeter groups are able to act on $V$ in this way. These questions  are completely solved when the associated bilinear form has no negative eigenvalues. The corresponding spherical and affine Coxeter groups are classified by the well-known Dynkin diagrams and their extensions. This can be found in \cite{MR1066460}.

We call a Coxeter group \emph{hyperbolic} when the standard action of $W$ on a metric vector space has only one negative eigenvalue. It is important to note that there are several competing meanings of the phrase ``hyperbolic Coxeter group" that exist in the literature.
This terminology makes sense since every Coxeter group of this type necessarily acts by isometries on hyperbolic space, although this action is not necessarily cocompact or even cofinite volume \cite{Kr10}.
We call hyperbolic Coxeter groups with a non-degenerate form  \emph{strongly hyperbolic} and ones with a degenerate form \emph{weakly hyperbolic} (this distinction is analogous to the dichotomy between spherical and affine).   A complete classification of hyperbolic Coxeter groups (in our sense) whose defining graph is a tree was completed collectively by Maxwell in \cite{MR511457}, Neumaier in \cite{MR664692} and Ottman in \cite{OttmanThesis}. A classification of arbitrary hyperbolic Coxeter groups is far from complete and the main theorem of this paper represents partial progress toward such a classification.

A \emph{higher rank Coxeter group} is one in which the associated form has at least two negative eigenvalues. The main theorem states that the defining graph of every higher rank Coxeter group must contain a pair of induced subgraphs with disjoint vertex sets, each of which defines a hyperbolic or higher rank Coxeter group. We go on to demonstrate that the converse to the main theorem is false by providing a counterexample, although we do prove a theorem which acts as a partial converse. Finally, we prove a collection of corollaries which identify  large classes of graphs that do not contain disjoint induced hyperbolic subgraphs, and, therefore, define hyperbolic, affine or spherical Coxeter groups. Although spherical and affine Coxeter groups have very restrictive Coxeter graphs, our main result and subsequent corollaries illustrate the rich variety in the Coxeter graphs of hyperbolic Coxeter groups.

\section{Preliminaries}

Given a Coxeter group $G=\langle s_1,...,s_n\mid (s_is_j)^{m_{ij}}\rangle$, define the \emph{Coxeter form matrix} to be the $n\times n$ matrix with entries $C_{ij}=-2\cos\frac{\pi}{m_{ij}}$. Define $V$ to be the $n$ dimensional real vector space with bilinear form $(x,y)=x^TCy$. A \emph{geometric representation space for $G$} is any $n-$dimensional metric vector space $W$ such that $G$ can be faithfully represented as a group generated by reflections on $W$. Notice that $V$ is a geometric representation space for $G$. The \emph{signature} of $G$ is the triple $(p,q,r)$ where $C$ has $p$ positive, $q$ negative, and $r$ zero eigenvalues.

Notice that if every $m_{ij}\in\{1,2,3\}$ then the associated Coxeter graph is a simple graph whose adjacency matrix is $2I-C$; in this case  $G$ is \emph{simply laced}. With this in mind, we define the \emph{generalized adjacency matrix $A$} for any Coxeter graph to be $2I-C$. Notice that if the eigenvalues of $A$ are  $\lambda_1\ge\lambda_2\ge...\ge\lambda_n$, then $G$ is spherical iff $\lambda_1<2$, $G$ is affine iff $\lambda_1=2$, $G$ is strongly hyperbolic iff $\lambda_1>2>\lambda_2$, $G$ is weakly hyperbolic iff $\lambda_1>2=\lambda_2$, and $G$ is higher rank iff $\lambda_2>2$. Also note that all entries of $A$ are nonnegative.

Given the Coxeter graph $\Gamma$ for $G$ and $X\subset\{s_1,...,s_n\}$, define the \emph{induced subgraph} $\Gamma_X$ to be the Coxeter graph associated to the Coxeter group $\langle s_i\in X\mid (s_is_j)^{m_{ij}}=1\rangle$. If $X,Y\subset S$, then $\Gamma_X$ and $\Gamma_Y$ are \emph{disjoint} if $X\cap Y=\emptyset$. They are \emph{separated} if $\Gamma_{X\cup Y}=\Gamma_X\sqcup\Gamma_Y$, i.e. there are no edges connecting a vertex in $X$ to a vertex in $Y$.

A matrix M is \emph{irreducible} if, for every $i,j$ there is some $k$ such that $(M^k)_{ij}\ne0$.
Close inspection will reveal that the generalized adjacency matrix $A$ is irreducible iff $\Gamma$ is connected.
The following is a classic theorem that we will state without proof. It can be found, for example, as Theorem $2.1.4(b)$ in \cite{MR1298430}.

\begin{thm}[Perron-Frobenius]\label{thm:perrfrob}
If $A$ is an irreducible $n\times n$ symmetric matrix with all non-negative entries, and eigenvalues $\lambda_1\ge\lambda_2\ge...\ge\lambda_n$ with corresponding eigenvectors $v_1,...,v_n$. Then $\lambda_1$ has multiplicity $1$, each coordinate of $v_1$ is positive and $v_i (i\ne 1)$ has at least one positive coordinate and at least one negative coordinate.
\end{thm}

The Perron-Frobenius Theorem can be applied to the generalized adjacency matrix  of any connected Coxeter graph $\Gamma$.

\section{A Decomposition Theorem for Higher Rank Graphs}

In this section, we show that the defining graph of every higher rank Coxeter group must contain two disjoint induced subgraphs which each define a hyperbolic Coxeter group.

\begin{defn}
Given a Coxeter graph $\Gamma$ with vertex set $\{s_1,...,s_n\}$ a \emph{label} of $\Gamma$ is a map $\phi: \{s_1,...,s_n\} \rightarrow \mathbb{R}$. Given a vector $v=[x_1,...,x_n]^T$, Let $\phi_v$ be the label such that $\phi_v(s_i)=x_i$.
\end{defn}

As mentioned above, we are interested in the eigenvalues of the generalized adjacency matrix $A$ of $\Gamma$. So
given a vector $v=[x_1,...,x_n]^T$ we would like to calculate $Av=[y_1,...,y_n]^T$. We can easily see that  the formula for $y_i$ is  $2x_i+\sum_{j=1}^n x_j(2\cos\frac{\pi}{m_{ij}})$. Notice that if $m_{ij}=1$ then $i=j$ and the $2x_i$ in front of the summation cancels the $x_i$ summand. Further, if $m_{ij}=2$ then $2\cos\frac{\pi}{m_{ij}}=0$, finally if $m_{ij}>2$ then $2\cos\frac{\pi}{m_{ij}}>0$. Therefore $y_i$ is a positive linear combination of all $x_j$ such that $s_j$ is adjacent to $s_i$ in $\Gamma$. Notice that if $G$ is simply laced then $y_i$ is the sum of all such $x_j$.

\begin{defn}
Suppose $\Gamma$ is a Coxeter graph with vertices $S=\{s_1,...,s_n\}$, $X\subset S$, and $\phi_v$ is a label of $\Gamma$ with $v=[x_1,...,x_n]^T$. If $\Gamma_X$ is the induced subgraph of $\Gamma$ with vertex set $X$, define the \emph{inherited label} $\phi_{v}^{X}$ of $\Gamma_X$ so that $\phi_{v}^{X}(s_i)=x_i$ for each $s_i\in X$.
\end{defn}

\begin{lem}\label{lem:passeigen}
If $M$ is a real symmetric matrix and there exists a vector $v=[x_1,...,x_n]^T$ with $Mv=[y_1,...,y_n]^T$ where $y_i\ge \mu x_i$ for each $i$, then there is an eigenvalue $\lambda$ of M such that $\lambda\ge\mu$.
\end{lem}

\begin{proof}
(Note: In this proof, $(,)$ is standard dot product on $\R^n$ and $|v|$ is the standard norm). Consider
\begin{align*}
(v,Mv)&=[x_1,...,x_n]\cdot[y_1,...,y_n]\\
&=x_1y_1+...+x_ny_n\\
&\ge \mu x_1^2+...+\mu x_n^2\\
&=\mu(x_1^2+...+x_n^2)\\
&=\mu|v|^2\\
\end{align*}

Given that we have a real symmetric matrix, we can find an orthonormal basis of eigenvectors $\{\omega_1,...,\omega_n\}$ with $M\omega_i=\lambda_i\omega_i$ and $\lambda_1\ge\lambda_2\ge...\ge\lambda_n$. Then we can write $v$ in this basis $v=b_1\omega_1+...+b_n\omega_n$, and we can see $Mv=b_1\lambda_1\omega_1+...+b_n\lambda_n\omega_n$. Again, consider the inner product
\begin{align*}
(v,Mv)&=b_1^2\lambda_1 (\omega_1,\omega_1)+...+b_n^2\lambda_n(\omega_n,\omega_n)\\
&=b_1^2\lambda_1+...+b_n^2\lambda_n\\
&\le \lambda_1(b_1^2+...+b_n^2)\\
&=\lambda_1|v|^2
\end{align*}
Combining these calculations, we have
\begin{align*}
\lambda_1|v|^2&\ge\mu|v|^2\\
\lambda_1&\ge\mu
\end{align*}
\end{proof}


\begin{thm}\label{thm:HRimpdisg}
If $\Gamma$ is a connected Coxeter graph whose associated Coxeter group is higher rank, then there exists disjoint induced subgraphs of $\Gamma$ which each define Coxeter groups that are hyperbolic or higher rank.
\end{thm}

\begin{proof}
Suppose we have $\Gamma$ with vertex set $S=\{s_1,...,s_n\}$. Suppose the generalized adjacency matrix $A$ of $\Gamma$ has eigenvalues  $\lambda_1\ge\lambda_2\ge...\ge\lambda_n$ and associated eigenvectors $v_1,...,v_n$. Since $\Gamma$ is higher rank, we know that $\lambda_1\ge\lambda_2>2$ and, by the Perron-Frobenius Theorem (\ref{thm:perrfrob}), $v_2=[x_1,...,x_n]^T$ has some positive and some negative coordinates. Examine the label $\phi_{v_2}$ of $\Gamma$ and define $P:=\{s_i|\phi_{v_2}(s_i)>0\}=\{s_{k_1},...,s_{k_p}\}$ and $N:=\{s_i|\phi_{v_2}(s_i)<0\}=\{s_{l_1},...,s_{l_q}\}$. Let $\phi_{v_2}^{P}$ and $\phi_{v_2}^{N}$ be the inherited labels of $\Gamma_P$ and $\Gamma_N$ respectively so that $\phi_{v_2}^{P}(s_{k_i})=x_{k_i}$ for $i\in \{1,...,p\}$ and $\phi_{v_2}^{N}(s_{l_i})=x_{l_i}$ for $i\in \{1,...,q\}$. Let $A_P$ be the generalized adjacency matrix of $\Gamma_P$ and let $A_N$ be the generalized adjacency matrix of $\Gamma_N$. Suppose $A_P[x_{k_1},...,x_{k_p}]^T=[w_1,...,w_p]^T$ and $A_N[-x_{l_1},...,-x_{l_q}]^T=[z_1,...,z_q]^T$.

\medskip

Claim:  $w_i\ge\lambda_2 x_{k_i}$ for $i\in \{1,...,p\}$ and $z_i\ge\lambda_2 (-x_{l_i})$ for $i\in \{1,...,q\}$.

\medskip

\emph{Proof of claim:} Since $Av_2=\lambda_2 v_2$, then

\begin{tabular}{ l l }
$\lambda_2 x_{k_i}$ & $=2x_i+\sum_{j=1}^n x_j(2\cos\frac{\pi}{m_{ij}})$\\
& $=\sum_{\{j|s_j\in P\}} |x_j(2\cos\frac{\pi}{m_{ij}})|-\sum_{\{j|s_j\in N\}} |x_j(2\cos\frac{\pi}{m_{ij}})|$\\
& $\leq \sum_{\{j|s_j\in P\}} |x_j(2\cos\frac{\pi}{m_{ij}})|$\\
& $=w_i.$
\end{tabular}

A similar argument shows $z_i\ge\lambda_2 (-x_{l_i})$ for $i\in \{1,...,q\}$.$\square$

\medskip

By Lemma \ref{lem:passeigen}, there must be some eigenvector for $A_P$ which has eigenvalue $\lambda\ge\lambda_2>2$. Hence, $\Gamma_P$ defines a hyperbolic or higher rank Coxeter group. Similarly, $\Gamma_N$ defines a hyperbolic or higher rank Coxeter group.

\end{proof}

The following corollary is the contrapositive of Theorem \ref{thm:HRimpdisg} with an additional connectedness condition added. It is particularly useful for demonstrating that certain Coxeter graphs are hyperbolic, affine or spherical.

\begin{cor}\label{cor:nohypdisimpnothr}
If $\Gamma$ is a connected Coxeter graph that contains no pair of connected disjoint induced subgraphs which each define a  hyperbolic or higher rank Coxeter group, then $\Gamma$ does not define a higher rank Coxeter group. Therefore, $\Gamma$ is spherical, affine or hyperbolic.
\end{cor}

\section{A Partial Converse}

In the previous section, we showed that any graph which defines a higher rank Coxeter group must have disjoint induced subgraphs which each define hyperbolic Coxeter groups. The next logical step is to try to prove the converse, that if Coxeter graph $\Gamma$ has disjoint induced subgraphs which each define hyperbolic Coxeter groups, then $\Gamma$ defines a higher rank Coxeter group. This is unfortunately not the case. A counterexample to this statement can be seen in Figure \ref{fig:counter_converse},
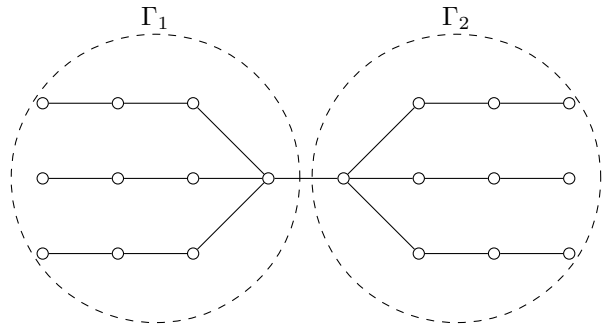
\begin{figure}
\begin{center}
\begin{tikzpicture}[inner sep=1.5, transform shape]

\node(v) at (0,0) [circle,draw]{};
\node(v11) at (1,1) [circle,draw]{};
\node(v12) at (2,1) [circle,draw]{};
\node(v13) at (3,1) [circle,draw]{};
\node(v21) at (1,0) [circle,draw]{};
\node(v22) at (2,0) [circle,draw]{};
\node(v23) at (3,0) [circle,draw]{};
\node(v31) at (1,-1) [circle,draw]{};
\node(v32) at (2,-1) [circle,draw]{};
\node(v33) at (3,-1) [circle,draw]{};
\node(u) at (-1,0) [circle,draw]{};
\node(u11) at (-2,1) [circle,draw]{};
\node(u12) at (-3,1) [circle,draw]{};
\node(u13) at (-4,1) [circle,draw]{};
\node(u21) at (-2,0) [circle,draw]{};
\node(u22) at (-3,0) [circle,draw]{};
\node(u23) at (-4,0) [circle,draw]{};
\node(u31) at (-2,-1) [circle,draw]{};
\node(u32) at (-3,-1) [circle,draw]{};
\node(u33) at (-4,-1) [circle,draw]{};

\draw[-] (u)--(v);
\draw[-] (v)--(v11);
\draw[-] (v11)--(v12);
\draw[-] (v12)--(v13);
\draw[-] (v)--(v21);
\draw[-] (v21)--(v22);
\draw[-] (v22)--(v23);
\draw[-] (v)--(v31);
\draw[-] (v31)--(v32);
\draw[-] (v32)--(v33);
\draw[-] (u)--(u11);
\draw[-] (u11)--(u12);
\draw[-] (u12)--(u13);
\draw[-] (u21)--(u);
\draw[-] (u21)--(u22);
\draw[-] (u22)--(u23);
\draw[-] (u)--(u31);
\draw[-] (u31)--(u32);
\draw[-] (u32)--(u33);

\node[draw,dashed,inner sep=.2pt,label=above:$\Gamma_2$,circle,fit=(v) (v11) (v31) (v13) (v33) ]{};
\node[draw,dashed,inner sep=.2pt,label=above:$\Gamma_1$,circle,fit=(u) (u11) (u31) (u13) (u33) ]{};

\end{tikzpicture} 
\end{center}
\caption{Counter Example to the Converse of Corollary \ref{cor:nohypdisimpnothr}}
\label{fig:counter_converse}
\end{figure}
$\Gamma_1$ and $\Gamma_2$ each define strongly hyperbolic Coxeter groups, by Theorem $2(a)$ of \cite{MR511457}, but the whole graph  also defines a  strongly hyperbolic Coxeter group, by Theorem $2(b)$ of \cite{MR511457}. Using these results, it is easy to construct an infinite family of graphs with this property. In fact, it is also  easy to see that the complete graph $K_n(n>3)$ is always strongly hyperbolic. With this fact in hand, we can see that the family of complete graphs  contain an arbitrarily large number of pairwise disjoint induced subgraphs that each define hyperbolic Coxeter groups. We quickly arrive at a partial converse if we instead focus on separated induced subgraphs.

\begin{thm}\label{thm:sepisimphr}
If $\Gamma$ is a Coxeter graph which contains separated induced subgraphs which each define hyperbolic Coxeter groups, then $\Gamma$ defines a higher rank Coxeter group.
\end{thm}

\begin{proof}
Suppose we have two separated induced subgraphs which each define a hyperbolic Coxeter group. Hence, the geometric representation spaces for those subgraphs each contain a vector with negative squared length such that these vectors are perpendicular in the geometric representation space for  $\Gamma$. Thus, the geometric representation space for $\Gamma$ has a $2$-dimensional negative definite subspace and therefore the Coxeter form matrix associated to $\Gamma$ has at least $2$ negative eigenvalues so $\Gamma$ is higher rank.
\end{proof}

The converse to the preceding theorem is also false. For example, the Coxeter graph pictured in Figure \ref{fig:parcon} defines a  higher rank Coxeter group, by Theorem $2(b)$ of \cite{MR511457}. However, it does not contain separated induced subgraphs which each define a hyperbolic Coxeter group.

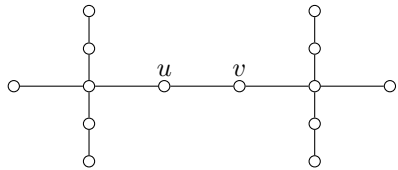
\begin{figure}
\begin{center}
\begin{tikzpicture}[inner sep=1.5]

\node(u1) at (1,0) [circle,draw]{};
\node(u2) at (0,0) [circle,draw]{};
\node(u3) at (1,.5) [circle,draw]{};
\node(u4) at (1,1) [circle,draw]{};
\node(u5) at (1,-.5) [circle,draw]{};
\node(u6) at (1,-1) [circle,draw]{};
\node(u) at (2,0) [circle,draw,label=above:$u$]{};
\node(v) at (3,0) [circle,draw,label=above:$v$]{};
\node(v1) at (4,0) [circle,draw]{};
\node(v2) at (5,0) [circle,draw]{};
\node(v3) at (4,.5) [circle,draw]{};
\node(v4) at (4,1) [circle,draw]{};
\node(v5) at (4,-.5) [circle,draw]{};
\node(v6) at (4,-1) [circle,draw]{};

\draw[-] (u)--(v);
\draw[-] (u)--(u1);
\draw[-] (u2)--(u1);
\draw[-] (u3)--(u1);
\draw[-] (u5)--(u1);
\draw[-] (u5)--(u6);
\draw[-] (u3)--(u4);
\draw[-] (v)--(v1);
\draw[-] (v2)--(v1);
\draw[-] (v3)--(v1);
\draw[-] (v5)--(v1);
\draw[-] (v5)--(v6);
\draw[-] (v3)--(v4);

\end{tikzpicture}
\end{center}
\caption{Counter Example to the Converse of Theorem \ref{thm:sepisimphr}}
\label{fig:parcon}
\end{figure}

The combined results of Theorem \ref{thm:HRimpdisg} and Theorem \ref{thm:sepisimphr} naturally lead to the following questions.

\begin{que}
Which graphs that define hyperbolic Coxeter groups contain disjoint induced subgraphs which define hyperbolic Coxeter groups?
\end{que}

\begin{que}
Which graphs that define higher rank Coxeter groups do not contain separated induced subgraphs which define hyperbolic Coxeter groups?
\end{que}

\section{Some Hyperbolic Coxeter Groups}

In the next corollary, we  reference  `the on-line encyclopedia of integer sequences' which can be found at {\url{http://www.research.att.com/~njas/sequences/Seis.html}}. Specifically, we will consider sequence A001349 which is the number of connected graphs with $n$ nodes. The relevant portion of the sequence is $n=1,...,7$ which is $1$, $1$, $2$, $6$, $21$, $112$, $853$. Using this sequence, we see that there are $4$ connected graphs with fewer than $4$ vertices, and there are $996$ connected graphs with fewer than $8$ vertices.

Similarly, we use the sequence A000055 which is the number of connected trees with $n$ nodes. The relevant portion of the sequence is on $n=1,...,11$ which is $1$, $1$, $1$, $2$, $3$, $6$, $11$, $23$, $47$, $106$, $235$. Using this sequence we see that there are $8$ connected trees with fewer than $6$ vertices and $436$ connected trees with fewer than $12$ vertices.

\begin{cor}
All $996$ connected unlabeled graphs with fewer than $8$ vertices and all $436$ connected unlabeled trees with fewer than $12$ vertices define  Coxeter groups which are hyperbolic, spherical, or affine.
\end{cor}

\begin{proof}
By the above comments, we know there are $4$ connected graphs with fewer than $4$ vertices. We can easily determine these to be $A_1, A_2, A_3,$ and $\widetilde{A}_2$ which are all spherical or affine.

Notice that if $\Gamma$ is an unlabeled graph with fewer than $8$ vertices and we consider a pair of disjoint induced subgraphs, one subgraph out of the pair must have fewer than $4$ vertices.
Therefore, we cannot find a pair of disjoint induced subgraphs which  each define a hyperbolic or higher rank Coxeter group. Hence, by Corollary \ref{cor:nohypdisimpnothr}, $\Gamma$ must define a hyperbolic, affine, or spherical Coxeter group.

Similarly, we know there are $8$ connected trees with fewer than $6$ vertices. Namely, these are $A_1$, $A_2$, $A_3$, $A_4$, $A_5$, $D_4$, $D_5$, and $\widetilde{D}_4$, which are all spherical or affine.

Similar to above, if $\Gamma$ is an unlabeled tree with fewer than $12$ vertices and we consider a pair of disjoint induced subgraph, one subgraph out of the pair must have fewer than $6$ vertices.
Therefore, we cannot find a pair of disjoint induced subgraphs which  each define  a hyperbolic or higher rank Coxeter group. Hence, by Corollary \ref{cor:nohypdisimpnothr}, $\Gamma$ must define a hyperbolic, affine, or spherical Coxeter group.
\end{proof}


\begin{defn}[Subhyperbolic Triple]
Suppose Coxeter graphs $G_1,G_2$ and $G_3$ which each define spherical or affine Coxeter groups are given. Furthermore, for each $G_k$, choose a vertex $v_k$, and labels $m_{12},m_{13},m_{23}\in\{2,...,\infty\}$. Define a \emph{subhyperbolic triple} (as in Figure \ref{fig:subhyperbolic_triple}) to be the Coxeter graph that is the disjoint union of $G_1$,$G_2$, and $G_3$ with an additional edge added between $v_i$ and $v_j$ with label $m_{ij}$ whenever $i\neq j$.
\end{defn}

\begin{figure}
\begin{center}
\begin{tikzpicture}[inner sep=1.5]

\node(v1) at (0,1) [circle, draw]{};
\node(v2) at ({cos(210)},{sin(210)}) [circle,draw]{};
\node(v3) at ({cos(330)},{sin(330)}) [circle,draw]{};

\draw[-] (v1)--node[above left]{$m_{12}$}(v2);
\draw[-] (v1)--node[above right]{$m_{13}$}(v3);
\draw[-] (v3)--node[below]{$m_{23}$}(v2);

\draw [dashed] (0,1.5) circle (.7cm){};
\node at (0,1.5) {$G_1$};

\draw [dashed] (1.5*{cos(210)},1.5*{sin(210)}) circle (.7cm){};
\node at (1.5*{cos(210)},1.5*{sin(210)}) {$G_2$};

\draw [dashed] (1.5*{cos(330)},1.5*{sin(330)}) circle (.7cm){};
\node at (1.5*{cos(330)},1.5*{sin(330)}) {$G_3$};



\end{tikzpicture} 
\end{center}
\caption{Subhyperbolic Triple}
\label{fig:subhyperbolic_triple}
\end{figure}
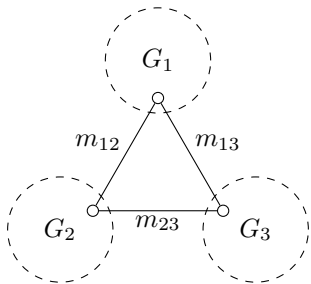

\begin{cor}\label{cor:subhyptrip}
Every Subhyperbolic Triple defines a Coxeter group which is hyperbolic, affine, or spherical
\end{cor}

\begin{proof}
If $\Gamma$ is a subhyperbolic triple, we will show it is impossible to find connected disjoint induced subgraphs which each define hyperbolic or higher rank Coxeter groups. First, notice that any connected induced subgraph which does not contain $v_1$, $v_2$, or $v_3$ must be a proper induced subgraph of $G_1$, $G_2$ or $G_3$. Since each of $G_1$, $G_2$ and $G_3$ is spherical or affine, any proper induced subgraph must define a spherical Coxeter group. Hence, for an induced subgraph of $\Gamma$ to define a hyperbolic or higher rank Coxeter group, it must contain $v_1$, $v_2$, or $v_3$.

Further, suppose $\Gamma_1$ is an induced subgraph containing only one of $v_1$, $v_2$, or $v_3$. With out loss of generality assume $\Gamma_1$ contains only $v_1$. Then $\Gamma_1$ is an induced subgraph of $G_1$ and, therefore, defines a spherical or affine Coxeter group. Thus, for an induced subgraph of $\Gamma$ define a hyperbolic or higher rank Coxeter group, it must contain at least two of $v_1$, $v_2$, or $v_3$. Since there are only three such vertices, it is impossible to have a pair of connected disjoint induced subgraphs that each define hyperbolic or higher rank Coxeter groups. So, by Corollary \ref{cor:nohypdisimpnothr}, $\Gamma$ defines a spherical, affine, or hyperbolic Coxeter group.
\end{proof}

\begin{cor}\label{cor:sphandaffplus1vertishyp}
If $\Gamma$ is a Coxeter graph with vertex set $S$ and there exists $s\in S$ such that $\Gamma_{S-\{s\}}$ is a disjoint collection of graphs that each define a spherical or affine Coxeter group, then $\Gamma$ defines a spherical, affine, or hyperbolic Coxeter group.
\end{cor}

\begin{proof}
Any connected induced subgraphs of $\Gamma$ that do not contain $s$ will define spherical or affine Coxeter groups, therefore, it is impossible to find connected disjoint induced subgraphs which each define hyperbolic or higher rank Coxeter groups. Hence, by Corollary \ref{cor:nohypdisimpnothr}, $\Gamma$ defines a spherical, affine, or hyperbolic Coxeter group.
\end{proof}

\bibliographystyle{plain}
\bibliography{bib}

\end{document}